\documentclass{birkmult}
 \newtheorem{thm}{Theorem}[section]
 
 \newtheorem{cor}[thm]{Corollary}
 \newtheorem{lem}[thm]{Lemma}
 \newtheorem{prop}[thm]{Proposition}
 \theoremstyle{definition}
 
 \theoremstyle{remark}
 \newtheorem*{Ack}{Acknowledgment}

 \numberwithin{equation}{section}
\newcommand{\D}{\mathbb{D}}
\newcommand{\Log}{\text{Log}}
\newcommand{\C}{\mathbb{C}}
\newcommand{\W}{W_{\psi, \phi}}
\renewcommand{\phi}{\varphi}
\begin{document}
\large
%-------------------------------------------------------------------------
% editorial commands: to be inserted by the editorial office
%
%\firstpage{1}
%\volume{228}
%\DOI{003-0001}
%\seriesextraline{This is the Concrete Title of this Book\br H.E. R and S.T.C. W, Eds.}
%\firstpage{1}
\title[Spaces supporting nontrivial Hermitian weighted composition operators
]
 {Reproducing kernel Hilbert spaces supporting nontrivial Hermitian weighted composition operators
}
%----------Author 1
\author[Bourdon]{Paul  Bourdon}

\address{%
Mathematics Department\\
Washington and Lee University\\
Lexington, VA 24450\\
USA}

\email{bourdonp@wlu.edu}

%\thanks{This work was completed with the support of our
%----------Author 2
\author[Shang]{Wenling Shang}
\address{Mathematics Department\\
Washington and Lee University\\
Lexington, VA 24450\\
USA}
\email{shangw11@mail.wlu.edu}
\subjclass{Primary 47B32; Secondary 47B33, 47B15}

\keywords{weighted composition operator, composition operator, weighted Hardy space, Hermitian operator. }

\date{March 1, 2011}

\begin{abstract}
  We characterize those  generating functions $k(z) = \sum_{j=0}^\infty z^j/\beta(j)^2$  that produce weighted Hardy spaces  $H^2(\beta)$ of the unit disk  $\D$ supporting nontrivial Hermitian weighted composition operators.   Our characterization shows that the spaces associated with the ``classical reproducing kernels''  $z\mapsto (1-\bar{w}z)^{-\eta}$,  where $w\in \D$ and $ \eta>0$,   as well as certain natural extensions of these spaces, are precisely those that are hospitable to Hermitian weighted composition operators.  It also leads to a refinement of a necessary condition for a weighted composition to be Hermitian, obtained recently by Cowen, Gunatillake, and Ko,  into one that is both necessary and sufficient.
\end{abstract}

%%% ----------------------------------------------------------------------
\maketitle
%%% ----------------------------------------------------------------------
%\tableofcontents

\section{Introduction}

     In broad terms, we show in this paper how an operator-theoretic condition determines a family of spaces that are potentially hospitable to operators satisfying the condition.  Every space that is potentially hospitable turns out to be entirely hospitable. 
   
    This paper concerns weighted composition operators acting on weighted Hardy spaces of the open unit disk $\D$ in the complex plane.   A Hilbert space comprising functions holomorphic on $\D$ in which the polynomials are dense and the monomials $1$, $z$, $z^2$, \ldots, constitute an orthogonal set of nonzero vectors is  a {\it weighted Hardy space}.  Each weighted Hardy space is characterized by its {\it weight sequence} $\beta$ defined by  $\beta(j) = \|z^j\|$ for  $j\ge 0$.  The weighted Hardy space $H^2(\beta)$ consists of those functions $f$ holomorphic on $\D$ whose Maclaurin coefficients $(\hat{f}(j))$ satisfy
$$
\sum_{j=0}^\infty |\hat{f}(j)|^2\beta(j)^2 < \infty.
$$
The inner product of $H^2(\beta)$ is given by
$$
\langle f, g\rangle = \sum_{j=0}^\infty \hat{f}(j)\overline{\hat{g}(j)}\beta(j)^2.
$$
If $\beta(j) = 1$ for all $j$, then $H^2(\beta)$ is the classical Hardy space $H^2$ of the disk; the choices $\beta(j) = (j+1)^{-1/2}$ and $\beta(j) = (j+1)^{1/2}$ yield, respectively, the classical Bergman and Dirichlet spaces of the disk.  {\em Throughout this paper we make the normalizing assumption that $\beta(0) = 1$.}

  The {\it generating function} $k$ of the weighted Hardy space $H^2(\beta)$ is defined by
$$
k(z) = \sum_{j=0}^\infty \frac{z^j}{\beta(j)^2}, \quad z\in \D.
$$
Thus, owing to our assumption that $\beta(0) =1$,  we have  $k(0)= 1$ for all generating functions.
It's not difficult to show  that $k$ must be analytic on $\D$ (\cite[Lemma 2.9]{CMB}) and that $k$ is a generating function for a weighted Hardy space $H^2(\beta)$  if and only if $\beta(j)$  is positive for each $j$, $\beta(0) =1$, and $\liminf \beta(j)^{1/j} \ge 1$ (see, e.g., exercise 2.1.10 of \cite{CMB}).   

The generating function $k$ generates reproducing kernels as follows.  For $w\in \D$, the function $K_w$ defined on $\D$ by $K_w(z) = k(\bar{w}z)$ is the reproducing kernel at $w$ for $H^2(\beta)$:

$$
\langle f, K_w\rangle = f(w)\quad  \text{for all}\  f \in H^2(\beta).
$$
  It's easy to see that for any $w\in \D$, the reproducing kernels $K_w^D$ and $K_w^{DD}$ for the first and second derivatives of functions in $H^2(\beta)$ are given, respectively, by 
\begin{equation}\label{RKD}
K^D_w(z) = zk'(\bar{w}z) \quad  \text{and}\quad  K_w^{DD}(z) = z^2k''(\bar{w}z),  \quad z\in \D.
\end{equation}

    Let $H(\D)$ denote the space of all functions holomorphic on $\D$.  Observe that whenever $\psi\in H(\D)$, and $\phi\in H(\D)$ has the property that it's a selfmap of $\D$ (i.e.,  $\phi(\D) \subseteq \D$), then  $W_{\psi,\phi}$ defined by $W_{\psi, \phi} f  = \psi\cdot (f\circ \phi)$ is a linear operator from $H(\D)$ to $H(\D)$.  We call $\psi$  and $\phi$ the {\it symbols} of the {\it weighted composition operator} $W_{\psi, \phi}$.    We consider weighted composition operators that restrict to be bounded on one or more weighted Hardy spaces.   Interest in weighted composition operators is growing, with some recent papers focusing on Hermitianness and normality  (see, e.g., \cite{BN}, \cite{CK}, \cite{CGK}), cyclicity (see, e.g., \cite{YR}), boundedness and compactness (see, e.g., \cite{AFC},   \cite{GalP}, \cite{GGC},  \cite{VM}), as well as invertibility and spectral behavior (see, e.g., \cite{BRJOT}, \cite{GG}, \cite{GGPP}).  Moreover, they arise naturally in  the context of other problems, for instance in that of characterizing adjoints of composition operators (see, e.g., \cite{BSA}, \cite{CG}, \cite{HMR}).   

    This paper is inspired by the following theorem \cite[Theorem 3]{CGK}, due to Cowen, Gunatillake, and Ko, that provides a necessary condition for $W_{\psi, \phi}$ to be Hermitian on $H^2(\beta)$.

\begin{thm}[Cowen,  Gunatillake, Ko]\label{CGK} Let $k$ be the generating function for $H^2(\beta)$.  If $W_{\psi,\phi}$ is a nonzero bounded Hermitian operator on $H^2(\beta)$, then $\psi(0)$  as well as  $\phi'(0)$ is real, and
\begin{equation}\label{HF}
\psi(z) = ck(\overline{a_0} z) \quad \text{while} \quad \phi(z) = a_0 + a_1 \beta(1)^2 z\frac{k'(\overline{a_0}z)}{k(\overline{a_0}z)},
\end{equation}
where $a_0 = \phi(0)$, $a_1 = \phi'(0)$, and $c = \psi(0)$.
\end{thm}

In \cite{CGK}, Cowen, Gunatillake, and Ko, show that the converse of the preceding theorem holds when $k(z) = (1-z)^{-\eta}$ and $\eta \ge 1$, the corresponding spaces being the standard-weighted Bergman spaces when $\eta > 1$ and the Hardy space $H^2$ when $\eta = 1$.  In this paper, we determine precisely when the converse holds.   There are three cases when the converse holds trivially.  
\begin{itemize}
\item $\psi$ is the zero function:  In this case, both $\W$ and $\W^*$ are the zero operator.
\item $\phi(0) = 0$:  Suppose that $a_0 = 0$, i.e., $\phi(0) = 0$,  and that  $\psi$ and $\phi$ are given by $(\ref{HF})$; then for all $z\in \D$, we have $\psi(z) =c$ (a real constant) and $\phi(z) = a_1z$ (where $a_1$ is real and $|a_1| \le 1$).   In this situation $\W$ is easily seen to be bounded (with norm $\le  |c|$) as well as  Hermitian on any weighted Hardy space $H^2(\beta)$. 
\item $\phi'(0) = 0$:   Suppose that $a_0$ is a point in $\D$,  that $a_1 = 0$ (i.e $\phi'(0)= 0$),  and that $\psi$ and $\phi$ are given by $(\ref{HF})$.   Then $\psi(z) = c k(\overline{a_0}z)$ and $\phi(z) = a_0$ for all $z\in \D$ and $c$ is real.   Notice that for any $f\in H^2(\beta)$, $\W f=  f(a_0) \psi = c\langle f, K_{a_0}\rangle K_{a_0}$,  so that the $\W$ is bounded on $H^2(\beta)$ with operator norm $\|\W\| \le |c|\|K_{a_0}\|^2$.  Moreover,    for any $w\in \D$,  we have  
$$
(\W K_w)(z) = ck(\overline{a_0}z)k(\bar{w}a_0)  \ \forall z\in \D
$$
and
$$
(\W^*K_w)(z)  = \overline{\psi(w)}K_{\phi(w)}(z) = ck(a_0\bar{w}) k(\overline{a_0}z)    \ \forall z\in \D.
$$
Thus, $\W^*$ and $\W$ agree when applied to each reproducing kernel and it follows that they agree when applied to each $f\in H^2(\beta)$.  
\end{itemize}
We are interested in determining  which weighted Hardy spaces support {\it nontrivial Hermitian weighted composition operators} $\W$, that is, those for which $\psi$ is not the zero function while both $\phi(0) \ne 0$ and $\phi'(0) \ne 0$.

 We completely characterize those weighted Hardy spaces supporting nontrivial Hermitian weighted composition operators.  We do this by deriving and solving a differential equation that the generating function $k$  for $H^2(\beta)$ must satisfy if the weighted composition operator $\W: H^2(\beta) \rightarrow H^2(\beta)$, whose symbols are defined by  (\ref{HF}),  is nontrivially Hermitian.   The result of our work is summarized in the following theorem.

\begin{thm}\label{MT}  Let $\phi\in H(\D)$ be a selfmap of $\D$ such that $\phi(0) \ne 0$ and $\phi'(0)\ne 0$. Let $\psi$ be an analytic function on $\D$ that is not the zero function.  The following are equivalent.
\begin{itemize}
\item[(i)]  $\W$ is a bounded Hermitian operator on $H^2(\beta)$.  
\item[(ii)]  $\W$ is a bounded operator on $H^2(\beta)$ satisfying 
$$
\W^* 1 = \W 1, \W^*z = \W z,\  \text{and}\  \W^*z^2 = \W z^2.
$$
\item[(iii)]  The generating function $k(z) = \sum_{j=0}^\infty z^j/\beta(j)^2$ for $H^2(\beta)$ takes one of the following forms
\begin{equation}\label{ocone}
k(z)=e^{\frac{z}{\beta(1)^2}}\quad \text{in case} \quad  \frac{2\beta(1)^4}{\beta(2)^2}=1,
\end{equation}
and otherwise,
\begin{equation}\label{ocnotone}
k(z)= (1-\lambda z)^{-\frac{1}{\lambda\beta(1)^2}},
\end{equation}
where $\lambda =  \frac{1}{\beta(1)^2}\left(\frac{2\beta(1)^4}{\beta(2)^2}-1\right)$
is a positive number less than or equal to $1$.  In addition,  the functions $\psi$ and $\phi$ satisfy
$$
\psi(z) = ck(\overline{a_0} z) \quad \text{while} \quad \phi(z) = a_0 + a_1 \beta(1)^2 z\frac{k'(\overline{a_0}z)}{k(\overline{a_0}z)},
$$
 where $a_0 = \phi(0)$, $a_1 = \phi'(0)$ is real, and $c = \psi(0)$ is real.
 \end{itemize}
\end{thm}

That (i) implies (ii) is trivial.   In the next section of this paper,  we show that (ii) implies (iii), adding to the necessary condition (\ref{HF}) for $\W$ to be Hermitian, the restrictions on $k$ described by (\ref{ocone}) and (\ref{ocnotone}).  In the final section, we prove that (iii) implies (i).  To show that if (iii) holds, then $\W$ is bounded on $H^2(\beta)$, we work with integral representations of norms of $H^2(\beta)$.  In particular, we exploit the fact that if its generating function $k$ has the form (\ref{ocone}), then all functions in $H^2(\beta)$ extend to be entire functions,  and $H^2(\beta)$ may be identified with the Fock space
\begin{equation}\label{FSF}
F^2_{1/\beta(1)^2}: = \{ f\in H(\C): \|f\|^2_{F} : = \frac{1}{\pi \beta(1)^2} \int_\C |f(z)|^2 e^\frac{-|z|^2}{
\beta(1)^2}\, dA(z) < \infty\},
\end{equation}
whose norm coincides with that of $H^2(\beta)$ (see, e.g., \cite{TG}).   
This Fock space may be viewed as the limit of the weighted Hardy spaces corresponding to the generating functions of $(\ref{ocnotone})$ as $\lambda$ approaches $0$.

\section{An additional necessary condition for $W_{\psi, \phi}$ to be Hermitian}

  Cowen, Gunatillake, and Ko's proof establishing that if $\W$ is a nonzero Hermitian operator on $H^2(\beta)$, then $\psi$ and $\phi$ have the formulas given by $(\ref{HF})$ is based on a comparison of $\W$ and $\W^*$ acting on reproducing kernels $K_w$.  One can also obtain this result from the equations  
$
\W^*1 = \W 1 \quad \text{and}\quad   \W^* z = \W z,
$
as the following proposition shows. 
\begin{prop}\label{FSTT}  Let $H^2(\beta)$ have generating function $k$.   Suppose that $\W$ is a nonzero bounded operator on $H^2(\beta)$ satisfying
\begin{equation}\label{FTHC}
\W^*1 = \W 1 \quad \text{and}\quad   \W^* z = \W z;
\end{equation}
then
$$
\psi(z) = ck(\overline{a_0} z) \quad \text{while} \quad \phi(z) = a_0 + a_1 \beta(1)^2 z\frac{k'(\overline{a_0}z)}{k(\overline{a_0}z)},
$$
 where $a_0 = \phi(0)$, $a_1 = \phi'(0)$ is real, and $c = \psi(0)$ is real.
\end{prop}
\begin{proof}  Observe that the constant function $z\mapsto 1$ is the reproducing kernel $K_0$ for $H^2(\beta)$.  Thus, our assumption that
$$
\W^* 1 = \W 1
$$
immediately yields  $\overline{\psi(0)} K_{\phi(0)} = \psi$.  Hence for each $z\in \D$, we have 
$$
\psi(z) = \overline{\psi(0)} k(\overline{a_0} z).
$$
Substituting $z = 0$ in the preceding equation (and remembering our standing assumption that $k(0) = 1$ for all generating functions), we get $\psi(0)= \overline{\psi(0)}$. Hence,  $\psi(0)$ is a real number $c$ and
$\psi(z) = c k(\overline{a_0}z)$, as desired.   Observe that $c$ must be nonzero for otherwise, $\W$ is the zero operator.  

  For each $g\in H^2(\beta)$, we have
  \begin{align*}
  \langle g, \W^* z\rangle & = \langle \psi\cdot g\circ \phi, z\rangle\\
  & =  (\psi\cdot g\circ \phi)'(0) \beta(1)^2\\
  & =  \langle g, \beta(1)^2\overline{\psi'(0)}K_{a_0} + \beta(1)^2\overline{\psi(0)}\overline{\phi'(0)}K_{a_0}^D\rangle.
  \end{align*}
Thus the equation $\W^*z = \W z$ yields
$$
\beta(1)^2\overline{\psi'(0)}K_{a_0} + \beta(1)^2\overline{\psi(0)}\overline{\phi'(0)}K_{a_0}^D = \psi \phi.
$$
Substituting $\psi(z) = ck(\overline{a_0}z)$, $\psi(0) = c$ (a nonzero real number), $\psi'(0) = c\overline{a_0}/\beta(1)^2$,  $K_{a_0}(z) = k(\overline{a_0} z)$, and $K_{a_0}^D(z) = z k'(\overline{a_0}z)$ into the preceding equation and solving for $\phi$, we obtain
$$
\phi(z) = a_0 + \overline{\phi'(0)}\beta(1)^2 z\frac{k'(\overline{a_0}z)}{k(\overline{a_0}z)}.
$$
Taking the derivative of both sides of this equation and letting $z = 0$ yields $\phi'(0) = \overline{\phi'(0)}$ making $\phi'(0)$ a real number $a_1$.  Thus $\phi$ has the desired form, completing the proof.
 \end{proof}

Proposition~\ref{FSTT} shows that the   equations $ W^*1 = \W 1$  and   $\W^* z = \W z$  of  Theorem~\ref{MT}, part (ii),  yield the restrictions on $\psi$ and $\phi$ expressed by the formulas (\ref{HF}).  We now show that the final equation of Theorem~\ref{MT}, part (ii), namely  $\W^* z^2 = \W z^2$, yields a restriction on $k$ expressed by a differential equation.

\begin{prop}\label{SSTT}  Let $k$ be the generating function for the weighted Hardy space $H^2(\beta)$.  Suppose that
$$
\phi(z) =  a_0 + a_1 \beta(1)^2 z\frac{k'(\overline{a_0}z)}{k(\overline{a_0}z)}
$$ 
is a selfmap of $\D$ and that
$$
\psi(z) =  ck(\overline{a_0} z)
$$
is a companion weight function such that 
$\W$ is a bounded operator on $H^2(\beta)$ satisfying
$$
\W^* z^2 = \W z^2
$$
 where we assume that $c$ and $a_1$ are real and nonzero.  Then for each $z\in \D$,
\begin{equation}\label{ODE1}
\beta(1)^4\frac{k^{\prime}(\overline{a_0}z)^2}{k(\overline{a_0}z)}=\frac{\beta(2)^2}{2}k^{\prime\prime}(\overline{a_0}z).
\end{equation}
\end{prop}
\begin{proof}  We assume that $\psi$ and $\phi$ are defined as in the statement of the theorem.    We derive the differential equation (\ref{ODE1}) from the equality
$$
W_{\psi, \phi}^* z^2 = \W z^2.
$$
 The right-hand side of the preceding equation is trivially evaluated to be $\psi \phi^2$; evaluating the left-hand side takes more effort.

Using the definition of the $H^2(\beta)$ inner product, we have for each $g\in H^2(\beta)$,
\begin{align*} 
\langle g, W^*_{\psi,\phi}z^2\rangle &= \langle \W g, z^2\rangle \\
                                                & = \frac{(\psi\cdot g\circ \phi)''(0)}{2} \beta(2)^2\\
                                                &=\frac{\beta(2)^2}{2} \left\langle g,c_1 K_{a_0} + c_2K_{a_0}^D + c_3K_{a_0}^{DD}\right\rangle,
                                                \end{align*}
                                                where $c_1 =  \overline{\psi''(0)}$, $c_2 = \left(\overline{\psi(0)}\overline{\phi''(0)} + 2\overline{\psi'(0)}\overline{\phi'(0)}\right)$, and $c_3 = \overline{\psi(0)} \overline{\phi'(0)^2}$.
  Thus,  
  $$
  \W^*z^2 =    \frac{\beta(2)^2}{2} \left(  \overline{\psi''(0)} K_{a_0} + \left(\overline{\psi(0)}\overline{\phi''(0)} + 2\overline{\psi'(0)}\overline{\phi'(0)}\right)K_{a_0}^D + \overline{\psi(0)} \overline{\phi'(0)^2}K_{a_0}^{DD}\right).
  $$
Substituting into the preceding equation, $\psi(0) = c$, $ \psi'(0) = \dfrac{c\overline{a_0}}{\beta(1)^2}$, $\psi''(0) = \dfrac{2c\overline{a_0}^2}{\beta(2)^2}$, $\phi'(0) = a_1$, and $\phi''(0) = \dfrac{4a_1\overline{a_0}\beta(1)^2}{\beta(2)^2} - \dfrac{2a_1 \overline{a_0}}{\beta(1)^2}$, as well the formulas for $K_{a_0}$, $K^D_{a_0}$, and $K^{DD}_{a_0}$  in terms of the generating function $k$, we obtain
  $$
    W^*_{\psi,\phi}z^2 =  ca_0^2k(\overline{a_0}z)   + 2c a_0 a_1 \beta(1)^2 zk'(\overline{a_0}z) + \frac{ca_1^2\beta(2)^2}{2}z^2k''(\overline{a_0}z).                        
$$
Now equate  $\W z^2$ and $\W^*z^2$, and use the formulas for $\psi$ and $\phi$ to obtain
\begin{equation*}
\begin{split}
 ck(\overline{a_0} z)\left(  a_0 + a_1 \beta(1)^2 z\frac{k'(\overline{a_0}z)}{k(\overline{a_0}z)}\right)^2  =  ca_0^2k(\overline{a_0}z)   + 2c a_0 a_1 \beta(1)^2 zk'(\overline{a_0}z) \\+ \frac{ca_1^2\beta(2)^2}{2}z^2k''(\overline{a_0}z),
 \end{split}
\end{equation*}
 which, because $a_1$ and $c$ are nonzero, simplifies to the desired differential equation that $k$ must satisfy at each point $z\in \D$: 
 $$
\beta(1)^4\frac{k^{\prime}(\overline{a_0}z)^2}{k(\overline{a_0}z)}=\frac{\beta(2)^2}{2}k^{\prime\prime}(\overline{a_0}z).
 $$
 
\end{proof}

We now solve the differential equation  (\ref{ODE1}) to find the form of generating functions for $H^2(\beta)$ spaces that are potentially hospitable to nontrivial Hermitian weighed composition operators.

\begin{thm}\label{NCMT} Let $k(z) = \sum_{j=0}^\infty z^j/\beta(j)^2$ be the generating function for $H^2(\beta)$, where $\beta(0) =1$.  Suppose that  $\W :H^2(\beta)\rightarrow H^2(\beta)$ is a nonzero bounded operator on $H^2(\beta)$, that $\phi(0)$ and $\phi'(0)$ are both nonzero, and that
$$
\W^* 1 = \W 1, \W^*z = \W z,\  \text{and}\  \W^*z^2 = \W z^2.
$$
Then $k$ must assume one of following two forms:
\begin{equation}\label{cone}
k(z)=e^{\frac{z}{\beta(1)^2}}\quad \text{in case} \quad  \frac{2\beta(1)^4}{\beta(2)^2}=1,
\end{equation}
and otherwise,
\begin{equation}\label{cnotone}
k(z)= (1-\lambda z)^{-\frac{1}{\lambda\beta(1)^2}},
\end{equation}
where $\lambda =  \frac{1}{\beta(1)^2}\left(\frac{2\beta(1)^4}{\beta(2)^2}-1\right)$
is a positive number less than or equal to $1$.  
\end{thm}
\begin{proof}  Thanks to Propositions~\ref{FSTT} and \ref{SSTT}, we know that under the hypotheses of the this theorem, $\psi$ and $\phi$ are given by (\ref{HF}) with $c$ as well as $a_1$ real, and $a_0$, $a_1$, and $c$ nonzero; we also know that $k$ satisfies (\ref{ODE1}), which means
\begin{equation}\label{ODE2}
\beta(1)^4\frac{k^{\prime}(z)^2}{k(z)}=\frac{\beta(2)^2}{2}k^{\prime\prime}(z), \quad \text{for all} \ z\in |a_0|\D.
\end{equation}
We solve the preceding equation subject to the initial conditions $k(0) =1$ and $k'(0) = 1/\beta(1)^2$.  

We begin by observing that because $k(0) = 1$ and $k'(0) = 1/\beta(1)^2 > 0$, there is an open disk $\D_1\subseteq |a_0|\D$ containing $0$ such that  both $k$ and $k'$ will have positive real part on $\D_1$.  We rewrite (\ref{ODE2}) as
$$
\frac{2\beta(1)^4}{\beta(2)^2} \frac{k'(z)}{k(z)}=\frac{k''(z)}{k'(z)}, \quad \text{for all} \ z\in \D_1.
$$
 Keeping in mind that $k$ and $k'$ have positive real part on $\D_1$, we see that the preceding equation yields
$$
\frac{2\beta(1)^4}{\beta(2)^2} \Log\, k = \Log\,  k' + C_1
$$
for some constant $C_1$, where the equation is valid on $\D_1$ and $\Log$ is the principal branch of the logarithm function.    We continue to work on $\D_1$ and
exponentiate both sides of the last equation;  upon setting 
$$
C_2 = e^{-C_1} \quad  \text{and}  \quad \gamma = \frac{2\beta(1)^4}{\beta(2)^2},
$$
we obtain
\begin{equation}\label{ODE3}
C_2=  k' e^{-\gamma \Log\, k} 
\end{equation}

We consider two cases,  the first being $\gamma=1$, so that (\ref{ODE3})  becomes
$$
C_2=\frac{k'}{k}.
$$
Upon integrating both sides of this equation, we obtain for some constant $C_3$ and all $z\in \D_1$
$$
C_2z+C_3= \Log\, k(z)
$$
or
$$
k(z) = e^{C_2z + C_3}.
$$
Using our initial conditions, $k(0) = 1$ and $k'(0) = 1/\beta(1)^2$, we obtain
$$
k(z) = e^{\frac{z}{\beta(1)^2}} \quad \text{for} \ z\in \D_1.
$$
Because each side of the preceding equation represents a function analytic on the entire unit disk $\D$, it follows that the equation holds on all of $\D$, yielding (\ref{cone}) from the statement of the theorem.  

For the second case, note $\gamma$ must be positive. We assume $\gamma \ne 1$ and integrate both sides of (\ref{ODE3}) to obtain that for some constant $C_3$ and all  $z\in \D_1$,
$$
C_2z + C_3 =  \frac{e^{(1-\gamma)\Log k(z)}}{1-\gamma}.
$$
Using our initial conditions, we obtain
$$ 
\left(1-\lambda z\right) = e^{(1-\gamma)\Log k(z)},  \quad z\in \D_1,
$$ 
where $\lambda:=-\frac{1-\gamma}{\beta(1)^2}$. 
Because $z\mapsto (1-\lambda z)$ takes the value $1$ at $0$,  we see that for all $z$ in some open disk containing $0$ and contained in $\D_1$, we must have
\begin{equation}\label{AEE}
e^{\frac{1}{1-\gamma} \Log(1-\lambda z)}= k(z).
\end{equation}
Because $k$ is analytic on $\D$, the left-hand side of (\ref{AEE}) extends to be analytic on $\D$ and agrees with $k$ at each of its points of analyticity within $\D$.  This means either  (a) $|\lambda|\le 1$ so that $z\mapsto (1-\lambda z)$ maps $\D$ into the right halfplane  or (b)  $1/(1-\gamma)$ is a positive integer.  However, (b) cannot hold because in this case $k$ would be a polynomial, contradicting the positivity of all of its Maclaurin coefficients.   Thus we conclude that $|\lambda| \le1$ and that
$$
k(z) = (1-\lambda z)^{\frac{1}{1-\gamma}} \quad \text{or}  \quad  k(z) = (1-\lambda z)^{-\frac{1}{\lambda \beta(1)^2}},
$$
where the root is the principal one.  Thus we have arrived at equation (\ref{cnotone}) in the statement of the theorem. All that remains is to show that $\lambda$, which equals $(\gamma - 1)/\beta(1)^2$, must be positive.     We know $\lambda$ is nonzero because $\gamma \ne 1$ in the present case.   If $\lambda$ were negative, the the binomial-series expansion of
$$
k(z) = (1-\lambda z)^{-\frac{1}{\lambda \beta(1)^2}}
$$
would contain some nonpositive coefficients contradicting the positivity of all $k$'s  Maclaurin coefficients.
\end{proof}

As a corollary of the proof of the preceding theorem, we obtain the following easily checked necessary condition for $H^2(\beta)$ to support nontrivial Hermitian weighted composition operators.
\begin{cor}  Let $k(z) = \sum_{j=0}^\infty z^j/\beta(j)^2$ be the generating function for $H^2(\beta)$, where $\beta(0) =1$.  If $H^2(\beta)$ supports a nontrivial Hermitian weighted composition operator, then
$$
0 \le \frac{1}{\beta(1)^2}\left( \frac{2\beta(1)^4}{\beta(2)^2} - 1\right) \le 1.
$$
\end{cor}
\begin{proof}   In the proof of Theorem~\ref{NCMT}, $\lambda =  \dfrac{1}{\beta(1)^2}\left( \frac{2\beta(1)^4}{\beta(2)^2} - 1\right)$ and it was shown that $0 \le \lambda \le 1$ must hold if $H^2(\beta)$ supports a nontrivial Hermitian weighted composition operator.
\end{proof}

Thus, e.g., the space $H^2(\beta_\omega)$ with $\beta_\omega(0) =1$ and $\beta_\omega(j) = 2$  for all $j \ge 1$ supports no nontrivial Hermitian weighted composition operators (because $\lambda = 7/4 > 1$).  We remark that the formulas (\ref{HF}) can yield a bounded weighted composition operator  on this space $H^2(\beta_\omega)$;  for instance, in defining $\psi$ and $\phi$ via (\ref{HF}), one may choose $a_0=1/2$, $a_1 = 1/10$ and $c=1$, making $\psi$ bounded on $\D$ and $\phi$ a selfmap of $\D$.  Observe that  $f$ belongs to the space $H^2(\beta_\omega)$  if and only if $f$ belongs to the classical Hardy space $H^2$ of $\D$; in fact,  for $f$ in these spaces,
\begin{equation}\label{WSHS}
 \|f\|_{H^2} \le \|f\|_{H^2(\beta_\omega)} \le 2 \|f\|_{H^2}.
\end{equation}
 Because $\psi$ is bounded and analytic on $\D$, the operator $M_\psi$ of multiplication by $\psi$ is bounded on $H^2$. The composition operator $C_\phi$ is also  bounded on $H^2$ (see, e.g., \cite[Chapter 3]{CMB}). Thus the weighted composition operator $\W = M_\psi C_\phi$ is bounded on $H^2$ and the inequalities (\ref{WSHS}) show the same is true of $\W$ on $H^2(\beta_\omega)$.

We have shown that in order for $\W$ to be a nontrivial bounded Hermitian operator on $H^2(\beta)$ then the symbols $\psi$ and $\phi$ not only have to satisfy the restrictions $(\ref{HF})$ obtained in \cite{CGK}, but also the generating function $k$ must have its form specified by (\ref{cone})  or (\ref{cnotone})  of Theorem~\ref{NCMT} above (equivalently  (\ref{ocone})  or (\ref{ocnotone})  of our main theorem, Theorem~\ref{MT}).  In the next section, we show these necessary conditions for $\W$ to be nontrivially Hermitian  are sufficient.

\section{The necessary condition is sufficient}

The goal of this section of the paper is to complete the proof of  Theorem~\ref{MT}  by establishing that part (iii) implies part (i).    There are two cases to consider.

\subsection{Case 1: The generating function $k$ for $H^2(\beta)$ has the form (\ref{ocone})}
In Case 1, we have
$$
k(z) = e^{\frac{z}{\beta(1)^2}},
$$
and the formulas of (\ref{HF}) yield
$$
\psi(z) = c e^{\frac{\overline{a_0}z}{\beta(1)^2}}\quad \text{and} \quad  \phi(z) = a_0 + a_1z,
$$
where $c$ and $a_1$ are nonzero real constants; also  $|a_1| < 1$ since $a_0\ne 0$ and $\phi$ is a selfmap of $\D$. In this case, the weight sequence for $H^2(\beta)$, which we will denote $\beta_F$, is given by
$$
\beta_F(j) = (j!)^{1/2}b^j, \quad \text{where} \quad  b = \beta(1).
$$
We have
\begin{equation}\label{MSF}
H^2(\beta_F) = \{ f\in H(\D):  \sum_{j=0}^\infty  |\hat{f}(j)|^2b^{2j} j! < \infty\},
\end{equation}
where $(\hat{f}(j))$ is the Maclaurin-coefficient sequence of $f$.  Note that each $f$ in $H^2(\beta_F)$ must have an analytic extension from $\D$ to the entire complex plane:  if $\sum_{j=0}^\infty  |\hat{f}(j)|^2b^{2j} j! < \infty$, then in particular there must be a positive constant $M$ such that $|\hat{f}(j)| \le \dfrac{M}{b^j (j!)^{1/2}}$ for all $j\ge 0$, and it follows that $\sum_{j=0}^\infty \hat{f}(j) z^j$ converges for every $z\in \C$.   A computation shows that $H^2(\beta_F)$ is the Fock space $F^2_{1/b^2}$ and that the $H^2(\beta_F)$ norm of $f$, which we denote $\|f\|_F$, is given by
$$
\|f\|^2_F = \frac{1}{\pi b^2}  \int_\C |f(z)|^2 e^\frac{-|z|^2}{b^2}\, dA(z).
$$
It's not difficult to show that $H^2(\beta_F)$, that is, $F^2_{1/b^2}$, does not support any  multiplication operators with nonconstant symbols.  However, it does support nontrivial bounded composition operators (see, e.g., \cite{CMS}) and weighted composition operators (see, e.g., \cite{UK}).  We could use a necessary and sufficient condition for boundedness of $\W$ on Fock spaces  appearing in \cite{UK} to obtain boundedness of our ``candidate'' Hermitian weighted composition operators, but a simple, direct proof is available.  

\begin{prop}\label{CoCa} Suppose that $\phi(z) = a_0 + a_1z$ where $0< |a_1| < 1$ and $\psi(z) = ce^{\frac{\overline{a_0}z}{b^2}}$.  Then $W_{\psi, \phi}$ is a bounded operator on $H^2(\beta_F)$.
\end{prop}
\begin{proof}  In preparation for a change of variable, note that $\phi^{-1}(z) = (z-a_0)/a_1$.  We have for each $f\in H^2(\beta_F)$,
\begin{eqnarray*}
\|W_{\psi, \phi} f\|_F^2  &=& \frac{|c|^2}{\pi b^2} \int_\C | e^\frac{\overline{a_0}z}{b^2} (f\circ \phi)(z)|^2 e^{-\frac{|z|^2}{b^2}}\, dA(z)\\
& \le & \frac{|c|^2}{\pi b^2} \int_\C |f(z)|^2 \exp\left(\frac{2|\overline{a_0}|||z-a_0|/|a_1| -|z-a_0|^2/|a_1|^2}{b^2}\right) \frac{1}{|a_1|^2}\, dA(z)\\
& \le  &  \frac{|c|^2}{|a_1|^2 }\sup\left\{\exp\left(\frac{2|\overline{a_0}|||z-a_0|/|a_1|+|z|^2 -|z-a_0|^2/|a_1|^2}{b^2}\right): z\in \C\right\}\|f\|_F^2.
\end{eqnarray*}
Observe that the supremum appearing in the last line of the preceding display is finite and the proposition follows. 

\end{proof}

Thus we see that in Case 1,  the operator $\W$ with symbols $\psi(z) = ce^{\overline{a_0}z/\beta(1)^2}$ and $\phi(z) = a_0 + a_1z$ is bounded on  the weighted Hardy space $H^2(\beta_F)$ having generating function $k(z) = e^{z/\beta(1)^2}$.   To see $ \W$ is Hermitian,  it suffices to show $\W^*K_w = \W K_w$ for all $w\in \D$. Let $w\in \D$ be arbitrary; we have for every $z\in \D$, 
\begin{eqnarray*}
(\W^* K_w)(z) & = & \overline{\psi(w)} K_{\phi(w)}(z)\\
& = & c\exp\left(\frac{a_0\bar{w}}{\beta(1)^2}\right) \exp\left(\frac{(\overline{a_0} + a_1\bar{w})z}{\beta(1)^2}\right)\\
& = & c \exp\left(\frac{\overline{a_0}z}{\beta(1)^2}\right) \exp\left(\frac{(a_0 + a_1z)\bar{w}}{\beta(1)^2}\right)\\
& = & (\W K_w)(z),
\end{eqnarray*}
as desired.

\subsection{Case 2: The generating function $k$ for $H^2(\beta)$ has the form (\ref{ocnotone})}
In Case 2, we have
$$
k(z) = (1-\lambda z)^{\frac{-1}{\lambda\beta(1)^2}},
$$
where $0 < \lambda \le 1$.  The  formulas of (\ref{HF}) yield
\begin{equation}\label{GNE}
\psi(z) = c (1-\lambda \overline{a_0}z)^{\frac{-1}{\lambda\beta(1)^2}} \quad \text{and} \quad  \phi(z) = a_0 + \frac{a_1z}{1-\lambda \overline{a_0} z}
\end{equation}
where $c$ and $a_1$ are nonzero real constants.  Of course, we continue to assume that  $\phi$ is a selfmap of $\D$ and $a_0  = \phi(0)$ is nonzero. {\em Note well that because $|a_0| < 1$, the function $\psi$ will be bounded on $\D$ for every $\lambda \in (0, 1]$.}

  Assuming for the moment that $\W$ is bounded, a routine computation shows $(\W^* K_w)(z) = (\W K_w)(z)$ for all $z\in \D$ with both sides simplifying to 
$$
c\left(\rule{0in}{0.15in} 1 - \lambda \overline{a_0}z - \lambda \overline{w}a_0(1-\lambda \overline{a_0}z) - \lambda \overline{w}a_1 z\right)^{\frac{-1}{\lambda\beta(1)^2}}.
$$
Thus $\W$ is Hermitian.  We now prove that  $\W$ acts boundedly on $H^2(\beta)$ in Case 2.

 Set $\eta = \frac{1}{\lambda\beta(1)^2}$  and for the moment let $\lambda =1$.   Thus we are interested in the weighted Hardy spaces $H^2(\beta_\eta)$ having generating functions of the form
 $$
 k_\eta(z) := (1-z)^{-\eta}, \quad \text{where}\  \eta > 0.
 $$
 In \cite{CGK}, Cowen et al.\ establish that whenever $\psi$ and $\phi$ are given by (\ref{GNE}) (with $\lambda  =1$), then $\W$ acts boundedly on $H^2(\beta_\eta)$ for $\eta \ge 1$.    The argument goes as follows.  Suppose that $\eta = 1$; then $H^2(\beta) = H^2$ is the classical Hardy space, so that if $f\in H^2(\beta_1)$, the square of its norm is $\frac{1}{2\pi} \int_0^{2\pi} |f(e^{it})|^2\, dt$, where $f$ is the radial limit function of $f$.    When $\eta >1$, $H^2(\beta)$ is the weighted Bergman space $L^2_a(\D, (1-|z|)^{\eta -2}\, dA(z))$  consisting of functions $f$ holomorphic on $\D$ for which 
 $$
 \int_\D |f(z)|^2(\eta-1)(1-|z|^2)^{\eta - 2}\, dA(z)/\pi < \infty;
 $$
 in fact, the integral on the left of preceding inequality equals the square of the $H^2(\beta_\eta)$  norm of $f$.    It is known that every analytic selfmap $\phi$ of $\D$ induces a bounded composition operator on the Hardy space $H^2$ as well as on the weighted Bergman spaces $L^2_a(\D, (1-|z|)^{\eta -2}\, dA(z))$, $\eta > 1$ (see, e.g., \cite[Chapter 3]{CMB}).  Moreover, the integral representations of the norms on these Hardy and Bergman spaces make it clear that whenever $\psi\in H(\D)$ is bounded on $\D$, then the multiplication operator $f\mapsto \psi f$ is bounded.  Thus, being the product of two bounded operators, $\W$ is bounded on $H^2(\beta_\eta)$ when $\eta \ge 1$.     We wish to show that $\W$ is bounded on $H^2(\beta_\eta)$ when $0 < \eta  <1$.   We depend on the following lemma.

  \begin{lem}\label{DerL} Let $0 < \eta < 1$. 
   The function $f$ belongs to $H^2(\beta_\eta)$ if and only if  $f'$ belongs to $H^2(\beta_{\eta+2})$, which is the weighted Bergman space
  $L^2_a(\D, (1-|z|)^{\eta}\, dA(z))$.   
  \end{lem}
  \begin{proof}
  Suppose that $f\in H^2(\beta_\eta)$, a space  having generating function
  $$
  k_\eta(z) = (1-z)^{-\eta} = \sum_{j=0}^\infty \frac{z^j}{\beta_\eta(j)^2}, 
  $$
  where
  $$
 \frac{1}{\beta_\eta(j)^2} = \frac{1}{j!}\prod_{m=0}^{j-1} (\eta + m).
  $$
  Observe that for $j\ge 0$,
  $$
 1\le  \frac{\beta_\eta(j+1)^2}{\beta_\eta(j)^2} = \frac{j+1}{\eta +j}\le \frac{1}{\eta}
  $$
  
  Via direct computation or through the observation that  
  $$
  k_{\eta+2} = \frac{1}{\eta(\eta+1)}k''_\eta(z) = \frac{1}{\eta(\eta+1)} \sum_{j=0}^\infty \frac{(j+2)(j+1)z^j}{\beta_\eta(j+2)^2},
  $$
   we have
$$
\frac{1}{\beta_{\eta+2}(j)^2} =   \frac{(j+1)(j+2)}{\eta(\eta+1)}\frac{1}{\beta_\eta(j+2)^2}.
$$
Thus  if  $f(z) = \sum_{j=0}^\infty \hat{f}(j)z^j$, then
\begin{align*}
\|f'\|^2_{H^2(\beta_{\eta+2})}  &= \sum_{j=1}^\infty j^2 |\hat{f}(j)|^2\beta_{\eta+2}(j-1)^2\\
& = \eta(\eta+1) \sum_{j=1}^\infty \frac{j}{j+1} |\hat{f}(j)|^2\beta_{\eta}(j+1)^2\\
& =    \eta(\eta+1) \sum_{j=1}^\infty \left(\frac{j}{j+1} |\hat{f}(j)|^2\beta_\eta(j)^2 \frac{\beta_{\eta}(j+1)^2}{\beta_\eta(j)^2}\right).
\end{align*}
Because $\frac{j}{j+1} \frac{\beta_{\eta}(j+1)^2}{\beta_\eta(j)^2}$ is bounded above by $1/\eta$ and below by $1/2$, the  preceding calculation of $\|f'\|^2_{H^2(\beta_{\eta+2})}$ establishes the lemma.
\end{proof}

\begin{prop}\label{BE1} Let $H^2(\beta_\eta)$ be the weighted Hardy space with generating function   $k_\eta(z) = (1-z)^{-\eta}$, where $\eta = 1/\beta(1)^2$. Suppose that $\psi$ and $\phi$ are given by $(\ref{GNE})$ with $k_\eta$ replacing $k$, which means
$$
\psi(z) = c (1- \overline{a_0}z)^{\frac{-1}{\beta(1)^2}} \quad \text{and} \quad  \phi(z) = a_0 + \frac{a_1z}{1- \overline{a_0} z},
$$
where $c$ and $a_1$ are real constants, $a_0  = \phi(0)$,  and $\phi$ is a selfmap of $\D$.
Then for every $\eta > 0$, $\W$ is a bounded operator on $H^2(\beta_\eta)$.
\end{prop}

\begin{proof}    We've already pointed out that $\W$ is bounded if $\eta \ge 1$.  Suppose that $0<\eta < 1$.   By the closed graph theorem, to establish the boundedness of $\W$ in general, we need only show that $\W f$ belongs to $H^2(\beta_\eta)$ whenever $f\in H^2(\beta_\eta)$.

Let $\psi$ and $\phi$ be as in the statement of the proposition.  Observe that $\psi$, $\psi'$, and $\phi'$ are all bounded on $\D$.  Let $f\in H^2(\beta_\eta)$ be arbitrary.  We show $\W f$ belongs to $H^2(\beta_\eta)$ by establishing that the derivative of $\psi f\circ \phi$ must belong to $L^2_a(\D, (1-|z|)^{\eta}\, dA(z))=H^2(\beta_{\eta+2})$.   Because $\beta_{\eta +2}(j) \le \beta_\eta(j)$  for all $j\ge 0$,  $H^2(\beta_\eta)\subseteq H^2(\beta_{\eta+2})=L^2_a(\D, (1-|z|)^{\eta}\, dA(z))$  and thus  $f$ belongs to $L^2_a(\D, (1-|z|)^{\eta}\, dA(z))$.  By Lemma~\ref{DerL}, $f'$ belongs to $L^2_a(\D, (1-|z|)^{\eta}\, dA(z))$.  Because $L^2_a(\D, (1-|z|)^{\eta}\, dA(z))$  is preserved by composition with any selfmap, both $f\circ \phi$ and $f'\circ \phi$ belong to $L^2_a(\D, (1-|z|)^{\eta}\, dA(z))$. Finally, because $\psi$, $\psi'$, and $\phi'$ are bounded on $\D$, we see $\psi\phi' f'\circ \phi + \psi' f\circ \phi = (\psi f\circ\phi)'$ belongs to $L^2_a(\D, (1-|z|)^{\eta}\, dA(z))$.  Thus $\W f$ belongs to $H^2(\beta_\eta)$ by Lemma~\ref{DerL}.  
\end{proof}

At this point we have shown that $\W$ is a bounded operator on  $H^2(\beta_\eta)$  whenever  $\eta >0$, where $\psi$ and $\phi$ are given by  (\ref{GNE}) with $\lambda =1$ and $\frac{1}{ \beta(1)^2} = \eta$.   To complete the proof that part (iii) implies part (i) of  Theorem~\ref{MT},  we must remove the restriction $\lambda = 1$.  

We depend on the following proposition, which is a generalization of \cite[Corollary 2.3]{CK}.

\begin{prop}\label{Dlem} Suppose that 
$$
\phi(z) = a_0 + \frac{a_1z}{1-\overline{a_0}\lambda z},
$$
where $a_1$ is real and $0 < \lambda \le 1$.   Let $\rho \le 1/\sqrt{\lambda}$ be a positive number such that $\rho|a_0|\lambda < 1$.  Then $\phi$ is a selfmap of the disk $\rho\D$ if and only if $|a_0| < \rho$ and
\begin{equation}\label{CSM}
 \frac{(1+|a_0|\lambda\rho)(|a_0| - \rho)}{\rho}  \le a_1 \le \frac{(\rho - |a_0|)(1-|a_0|\lambda\rho)}{\rho}.
\end{equation}
\end{prop}
\begin{proof}  If $a_0 = 0$ or if $a_1 = 0$, then the proposition is easily seen to hold.  Suppose that both $a_0$ and $a_1$ are nonzero.  In this case,  the estimates on  $a_1$  given by  (\ref{CSM})  cannot hold unless $|a_0| < \rho$.  Thus, to complete the proof of the proposition, we show that  $\phi$ is s selfmap of $\rho \D$ (with $a_0$ and $a_1$ nonzero) if and only if (\ref{CSM}) holds.   

The assumption that $\rho < 1/(|a_0|\lambda)$ means $\phi$ is analytic on the closed disk $\rho\D^-$.  Our assumption that $a_1\ne 0$, means that $\phi$ is nonconstant.   Thus by the Maximum-Modulus Theorem, for every $z\in \rho\D$,
\begin{align}
|\phi(z)|&  < \max_{t\in [0, 2\pi)} \left| a_0 + \frac{a_1 \rho \frac{a_0}{|a_0|} e^{it}}{1-|a_0|\lambda\rho e^{it}}\right| \nonumber\\
  & = \max_{t\in [0, 2\pi)} \left| |a_0| + \frac{a_1 \rho e^{it}}{1-|a_0|\lambda\rho e^{it}}\right|. \label{MMTC}
  \end{align}
 Let $\Gamma: [0, 2\pi)\rightarrow \C$ be given by $\Gamma(t) = |a_0| +  (a_1 \rho e^{it})(1-|a_0|\lambda\rho e^{it})^{-1}$ and observe that the image of $\Gamma$ is a circle with diameter joining the real numbers
 $
|a_0| -  (a_1 \rho)(1+|a_0|\lambda\rho)^{-1}$   and  $ |a_0| +  (a_1 \rho )(1-|a_0|\lambda\rho)^{-1}$. Thus, the maximum of (\ref{MMTC}) is attained at when either $e^{it} =1$ or $e^{it} =-1$, and hence
 $\phi$ is a selfmap of $\rho D$ if and only if 
  \begin{equation}\label{INEQ1}
  -\rho \le |a_0| -  \frac{a_1 \rho}{1+|a_0|\lambda\rho}  \le \rho
    \end{equation}
  and
   \begin{equation}\label{INEQ2}
  -\rho \le |a_0| + \frac{a_1 \rho }{1-|a_0|\lambda\rho}  \le \rho.
  \end{equation}
 The inequalities on the right of  (\ref{INEQ1}) and (\ref{INEQ2}) yield, respectively, the inequalities on the left and right of (\ref{CSM}).  Moreover, the inequalities on the left of (\ref{INEQ1}) and (\ref{INEQ2}) are automatically satisfied if  (\ref{CSM}) holds (where that on the left of  (\ref{INEQ2})  is satisfied owing to our assumption that $\rho \le 1/\sqrt{\lambda}$).  These observations complete the proof.     \end{proof}

\begin{cor}\label{SMC}  Suppose that $a_1$ is real, that $0 < \lambda \le 1$, and that
$$
\phi(z) = a_0 + \frac{a_1z}{1-\overline{a_0}\lambda z}
$$
is a selfmap of $\D$; then $\phi$ is also a selfmap of $\frac{1}{\sqrt{\lambda}}\D$.
\end{cor}
\begin{proof}  Applying Proposition~\ref{Dlem} with $\rho =1$, we see $|a_0| < 1$ and
\begin{equation}\label{FCE}
 (1+|a_0|\lambda)(|a_0| - 1)   \le a_1 \le (1 - |a_0|)(1-|a_0|\lambda).
\end{equation}
To complete the proof,  we apply Proposition~\ref{Dlem} with $\rho = 1/\sqrt{\lambda}$.  Clearly $|a_0| < 1/\sqrt{\lambda}$ and we complete the proof by showing $a_1$ satisfies (\ref{CSM}) with $\rho =1/\sqrt{\lambda}$.  Note the when $\rho =1/\sqrt{\lambda}$, the rightmost quantity in (\ref{CSM}) is $(1-\sqrt{\lambda}|a_0|)^2$  and the leftmost quantity is $\lambda|a_0|^2 -1$.  We have
\begin{align*}
a_1& \le  (1 - |a_0|)(1-|a_0|\lambda)\quad (\text{by}\ (\ref{FCE}))\\
   & =1 -(1+\lambda)|a_0| + |a_0|^2\lambda\\
   & \le 1 -2\sqrt{\lambda}|a_0| + |a_0|^2\lambda\\
   & = (1-\sqrt{\lambda}|a_0|)^2,
   \end{align*}
   as desired.
   Also,
   \begin{align*}
   a_1 & \ge (1+|a_0|\lambda)(|a_0| - 1)\quad (\text{by}\ (\ref{FCE}))\\
        & =  \lambda |a_0|^2 - 1 +|a_0|(1-\lambda)\\
     & \ge \lambda |a_0|^2 - 1,
     \end{align*}
     as desired.
     \end{proof}

     \begin{prop}  Let $0< \lambda\le 1$, let $\eta = \frac{1}{\lambda \beta(1)^2}$,  and let $H^2_\lambda(\beta_\eta)$ be the weighted Hardy space with generating function $k(z) = (1- \lambda z)^{-\frac{1}{\lambda\beta(1)^2}}$.  Suppose that $\psi$ and $\phi$ are given by $(\ref{GNE})$, which means
      $$
      \psi(z) = c (1- \overline{a_0}\lambda z)^{\frac{-1}{\lambda\beta(1)^2}}\quad \text{and}  \quad  \phi(z) = a_0 + \frac{a_1z}{1-\overline{a_0}\lambda z}, 
      $$
 where $c$ and $a_1$ are real constants, $a_0  = \phi(0)$,  and $\phi$ is a selfmap of $\D$.  Then  $\W$ is bounded on $H_\lambda^2(\beta_\eta)$.
 \end{prop}
 \begin{proof} Because $\phi$  is a selfmap of $\D$, Corollary~\ref{SMC} shows that  $z\mapsto \phi(z/\sqrt{\lambda})$ maps $\D$ into $\frac{1}{\sqrt{\lambda}}\D$.   Hence $z\mapsto \sqrt{\lambda} \phi(z/\sqrt{\lambda})$ is a selfmap of $\D$, which we denote $\tilde{\phi}$:
     $$
     \tilde{\phi}(z) = \sqrt{\lambda} \phi(z/\sqrt{\lambda}) = \sqrt{\lambda}a_0 + \frac{a_1z}{1-\overline{a_0}\sqrt{\lambda}z}.
     $$

        As above, let $H^2(\beta_\eta)$ be the weighted Hardy space with generating function $k_\eta(z) = (1-z)^{-\eta}$.  Observe that any function $f$ in $H^2_\lambda(\beta_n)$ has analytic extension to the disk $\frac{1}{\sqrt{\lambda}}\D$.   Let $u(z) = \sqrt{\lambda} z$.  Observe that the composition operator $C_u: H^2(\beta_\eta) \rightarrow H^2_\lambda(\beta_\eta)$  is a surjective isometry---that is, it is a unitary operator with adjoint $C_\nu$, where $\nu(z) = z/\sqrt{\lambda}$.    Define $\tilde{\psi}(z) = ck_\eta(\sqrt{\lambda}\, \overline{a_0}z)$.  The weighted composition operator $W_{\tilde{\psi},\tilde{\phi}}$ is bounded on $H^2(\beta_\eta)$ by Proposition~\ref{BE1}.    Observe that $C_u W_{\tilde{\psi},\tilde{\phi}} C_\nu$ is the weighted composition operator $\W$ on $H^2_\lambda(\beta_\eta)$ and thus $\W$ is bounded on $H^2_\lambda(\beta_\eta)$; in fact, it's unitarily equivalent to the bounded operator $W_{\tilde{\psi}, \tilde{\phi}}$ on $H^2(\beta_\eta)$.
     \end{proof}
     
    The preceding proposition completes our proof of Theorem~\ref{MT}, showing when part (iii) holds and  $k$ is given by (\ref{ocnotone}), then $\W$ is bounded on $H^2(\beta)$.  (Proposition~\ref{CoCa} establishes the analogous result when $k$ is given by (\ref{ocone}).)
    
    \begin{Ack}  The authors wish to thank Joel H.\ Shapiro for helpful comments, references, and suggestions,  which improved this paper.
    \end{Ack}

\end{document}